\documentclass[12pt]{amsart}

\usepackage{graphicx}
\usepackage{amsthm}
\usepackage{amsmath, amssymb}
\swapnumbers

\theoremstyle{plain}
\newtheorem{Thm}{Theorem}

\newtheorem{Coro}[Thm]{Corollary}
\newtheorem{Lem}[Thm]{Lemma}
\newtheorem{Claim}{Claim}

\theoremstyle{definition}
\newtheorem{Def}[Thm]{Definition}

\newcommand{\mS}{\mathcal{S}}
\newcommand{\mB}{\mathcal{B}}

\begin{document}

\title{One-sided and two-sided Heegaard splittings}

\author{Jesse Johnson}
\address{\hskip-\parindent
        Department of Mathematics \\
        Oklahoma State University \\
        Stillwater, OK 74078
        USA}
\email{jjohnson@math.okstate.edu}


\thanks{Research supported by NSF grant DMS-1006369}

\begin{abstract}
We define a notion of Hempel distance for one-sided Heegaard splittings and show that the existence of alternate surfaces restricts distance for one-sided splittings in a manner similar to Hartshorn's and Scharlemann-Tomova's results for two-sided splittings. We also show that every geometrically compressible one-sided Heegaard surface in a non-Haken 3-manifold is stabilized, and show that the mapping class group of the two-sided Heegaard splitting induced by a high distance one-sided splitting is isomorphic to the fundamental group of the one-sided surface.
\end{abstract}

\maketitle

Let $M$ be a compact, connected, closed, orientable 3-manifold.  A surface $O$ embedded in $M$ is \textit{one sided} if a regular neighborhood $U$ of $O$ is homeomorphic to a twisted interval bundle, i.e. a bundle of intervals over $O$ that is not a product $O \times [0,1]$.  A \textit{one sided Heegaard surface} for $M$ is a one-sided surface $O \subset M$ such that the complement in $M$ of the regular neighborhood $U$ is a handlebody.  We can also extend this handlebody into $U$ and think of it as an immersed handlebody with embedded interior whose boundary double covers $O$.

Hempel~\cite{hempel} defined the \textit{distance} of a two-sided Heegaard surface $\Sigma$ as the distance in the curve complex for $\Sigma$ between the sets of loops bounding disks in either handlebody. We will define the \textit{distance} $d(O)$ for a one-sided Heegaard surface as the distance in the curve complex for the surface $S = \partial U$ from the set of loops bounding disks in the complement of $U$ to the set of loops that are in the boundary of an essential annulus or M\"obious band in $U$.

One goal of this paper is to show that this is a good notion of distance for a one-sided Heegaard surface. We construct one-sided splittings with arbitrarily high distance by mimicking the corresponding construction for two-sided splittings.  However, the main reason that this appears to be the correct notion of distance is an analogue of Hartshorn's~\cite{hartshorn} and Scharlemann-Tomova's~\cite{st:dist} ``distance bounds genus'' results:

\begin{Thm}
\label{mainthm1}
If $O$ is a one-sided Heegaard surface with distance $d(O)$ and $T$ is either an incompressible two-sided surface or a one-sided Heegaard surface with Euler characteristic $\chi(T)$ then either $d(O) \leq \chi(T)$ or $T$ is a stabilization of $O$.
\end{Thm}

By a \textit{stabilization} of a one-sided Heegaard surface, we mean the following: Given a closed ball $B \subset M$ that intersects $O$ in a single disk, let $O'$ be the surface that result from replacing the disk $O \cap B$ with an unknotted once-punctured torus. The resulting surface is still one sided and the complement of its regular neighborhood will be a handlebody of genus two greater than the genus of the original handlebody.  The surface $O'$ and any surface that comes from repeating this construction is called a \textit{stabilization} of $O$.  Any one-sided Heegaard surface that results from stabilizing a (different) one-sided Heegaard surface is called \textit{stabilized}.

A key step in proving Theorem~\ref{mainthm1} is the following Lemma, which was conjectured by Bartolini and Rubinstein:

\begin{Lem}
\label{mainlem}
If $O$ and $O'$ are one-sided Heegaard surfaces such that  $O'$ results from repeatedly compressing $O$ then $O$ is a stabilization of $O'$.
\end{Lem}

Note that the inclusion map from $\pi_1(O)$ to $\pi_1(M)$ will have a very large kernel, so $O$ cannot be incompressible in the algebraic sense.  By a compression, we mean a geometric compression: we will say that $O$ is \textit{geometrically compressible} if there is an embedded disk in $M$ whose boundary is an essential simple closed curve in $O$. Rubinstein and Bartolini~\cite{bart2}\cite{bart4}\cite{bart3}\cite{bart1}\cite{Ru} have made significant progress understanding geometrically incompressible one-sided Heegaard surfaces by playing off both their incompressible and their compressible nature. Birman-Rubinstein~\cite{RuBi} used one-sided surfaces to make some of the first computations of homeotopy (mapping class) groups of certain 3-manifolds.

Every stabilized one-sided Heegaard surface is geometrically compressible. In a non-Haken 3-manifold, every geometrically incompressible one-sided surface is a one-sided Heegaard surface and the converse is true as well:

\begin{Coro}
\label{compcoro}
Every geometrically compressible one-sided Heegaard surface in a non-Haken 3-manifold is stabilized.
\end{Coro}

From any one-sided Heegaard surface $O$, we can construct a two-sided Heegaard surface $\Sigma'$ by attaching a vertical tube to the doubled surface $\Sigma = \partial U$ inside the interval bundle $U$. This Heegaard surface has distance at most 2, so the standard results about high distance Heegaard splittings (such as~\cite{st:dist}) don't apply to it. However, we will use the distance of the one-sided splitting to understand $\Sigma'$ as follows:

\begin{Thm}
\label{2sidedthm1}
If $\Sigma'$ is induced from a one-sided Heegaard surface $O$ then every genus $g$ Heegaard surface $T \subset M$ with $2g \leq d(O)$ is a stabilization of $\Sigma'$.
\end{Thm}

A similar Theorem (with a slightly weaker bound) follows from Tao Li's much more general results about gluing manifolds together along their boundaries~\cite{li}. However, we will develop techniques here that also allow us to understand the mapping class group of the induced Heegaard splitting.

The \textit{mapping class group} $Mod(M, \Sigma')$ is the group of self-homeomomorphisms of $M$ that take $\Sigma'$ onto itself, modulo isotopies of $M$ that preserve $\Sigma'$ setwise. The Heegaard surfaces constructed this way form one family in a relatively short list of types of Heegaard surfaces with irreducible automorphisms classified by the author and Hyam Rubinstein~\cite{jr:mcgs}. Because the tube used to construct $\Sigma'$ from $\Sigma$ can be attached along any vertical interval in $U$, elements of $Mod(M, \Sigma')$ can be constructed by dragging this tube around a path in $O$. So $Mod(M, \Sigma')$ contains a subgroup isomorphic to $\pi_1(O)$. If the distance of $O$ is sufficiently high, we will show that this is the entire mapping class group:

\begin{Thm}
\label{2sidedthm2}
If $\Sigma'$ has genus $g$ and is induced by a one-sided Heegaard surface $O$ with $2g < d(O)$ then the mapping class group $Mod(M, \Sigma')$ is isomorphic to the fundamental group $\pi_1(O)$.
\end{Thm}

The paper is organized as follows: In Section~\ref{distsect}, we give a more detailed description of the distance of a one-sided Heegaard surface and prove that there are one-sided Heegaard surfaces of arbitrarily high distance. In Section~\ref{flatsect}, we review the definition of flat surfaces, as introduced in~\cite{me:upperbnd}.  Sweep-outs are reviewed in Section~\ref{sweepsect}, and used to prove Theorem~\ref{mainthm1}. Lemma~\ref{mainlem} is proved in Section~\ref{surfacecplxsect}. A construction called a \textit{band move} is introduced in Section~\ref{twosidesect} and used to prove  Theorem~\ref{2sidedthm1} in Section~\ref{twosideboundsect}. Finally, Theorem~\ref{2sidedthm2} is proved in Section~\ref{mcgsect}.

I thank Loretta Bartolini for introducing me to one-sided Heegaard surfaces and dissuading me of my early, naive notions about them. I also thank Saul Schleimer for relating to me the details of constructing high distance one-sided Heegaard surfaces.

\section{Distance}
\label{distsect}

Given a one-sided Heegaard surface $O \subset M$, let $U$ be a closed regular neighborhood of $O$ and $\Sigma = \partial U$. The set $U$ is an interval bundle such that $O$ consists of the midpoints of the intervals and $\Sigma$ consists of the endpoints of the intervals. The map that sends the endpoints of each interval to the midpoint defines a covering map $c : \Sigma \rightarrow O$. (This is the orientation cover of the non-orientable surface $O$.)

The pre-image in $c$ of every simple closed curve in $O$ is either one or two loops in $\Sigma$. Let $\mathcal{O} \subset \mathcal{C}(\Sigma)$ be the set of loops that are in the pre-images of loops in $O$, where $\mathcal{C}(\Sigma)$ is the curve complex. The closure of the complement of $N$ is a handlebody whose boundary is also $\Sigma$. Define $\mathcal{H} \subset \mathcal{C}(\Sigma)$ as the set of simple closed curves that are boundaries of disks in $H$. Define the \textit{distance} of $O$ as $d(O) = d(\mathcal{O}, \mathcal{H})$, i.e.\ the shortest edge path distance in the curve complex from a loop in one set to a loop in the other.

We will see below that this distance has similar properties to the distance for a two-sided splitting. For example, note that if $O$ is geometrically compressible then the geometric compression defines an annulus in $U$ and a disk in $H$ with a common boundary loop so $d(O) = 0$. Conversely, we will check below that there are splittings of arbitrarily high distance.

The construction of high distance one-sided splittings is very similar to the construction of high distance two-sided splittings~\cite[Theorem 2.7]{hempel}, which Hempel attributes to Feng Luo and Luo attributes to Tsuyoshi Kobayashi~\cite{kob:loops}. The proof in this setting is complicated by the fact that we are measuring distance from the set of doubles of loops in $O$, but it still uses fairly standard machinery from Thurston's approach to surface automorphisms~\cite{thurston} and was transmitted to me by Saul Schleimer. Because this is not the main focus of the present paper, we will only outline the argument as a series of claims, with proofs left to the reader.

For the definitions of measured lamination, the spaces $ML(O)$, $ML(\Sigma)$ of measured laminations, the projective measured space spaces $PML(O)$, $PML(\Sigma)$, and other terminology, see Penner and Harer's book~\cite{penhar}. 

We have defined the set $\mathcal{O}$ in the curve complex for $\Sigma$ as the set of preimages of simple closed curves in $O$ under the covering map $c : \Sigma \rightarrow O$. We would like to lift laminations from $O$ to $\Sigma$ in a similar way.

\begin{Claim}
There is a well defined lifting map $c'_* : ML(O) \rightarrow ML(\Sigma)$ that is continuous, piecewise-linear, injective and equivariant under scaling of measures.
\end{Claim}

Because $c'_*$ is equivariant under scaling of measures, it defines a map between the projective measured lamination spaces.

\begin{Claim}
The induced map $c : PML(O) \rightarrow PML(\Sigma)$ is continuous, piecewise-linear and injective and its image is the closure of $\mathcal{O}$.
\end{Claim}

Masur~\cite{masur} has shown that the closure in $PML(\Sigma)$ of the set $\mathcal{H}$ of boundaries of essential disks in $H$ has empty interior in $PML(\Sigma)$. By calculating the dimensions of $PML(\Sigma)$ and $PML(O)$, we find:

\begin{Claim}
If $O$ is not a projective plane or Klein bottle then the image in $PML(\Sigma)$ of $PML(O)$ has empty interior.
\end{Claim}

On the other hand, the set of stable and unstable laminations of pseudo-Anosov maps on $S$ is dense in $PML(\Sigma)$. We thus have:

\begin{Claim}
There is a pseudo-Anosov map $\phi : S \rightarrow S$ whose stable and unstable laminations are both disjoint from the closures of $\mathcal{O}$ and $\mathcal{H}$
\end{Claim}

As in~\cite{hempel}, we will compose the gluing map for a given one-sided Heegaard splitting with a sufficiently high power of a map $\phi$ whose stable and unstable laminations are disjoint from the closures in $PML(\Sigma)$ of $\mathcal{H}$ and $\mathcal{O}$.

\begin{Lem}
\label{exhighdistlem}
For every positive integer $K$, there is a 3-manifold $M'$ and a one-sided Heegaard surface $O' \subset M'$ such that $d(O') > K$.
\end{Lem}

\section{Flat surfaces}
\label{flatsect}

In this section, we review the definitions of flat and essential flat surfaces from~\cite{me:upperbnd}.  Let $\Sigma$ be a compact, connected, closed, orientable surface and let $N = \Sigma \times [0,1]$.  A \textit{vertical annulus} in $N$ is a surface of the form $\ell \times [a,b]$ where $\ell$ is a simple closed curve in $\Sigma$ and $[a,b]$ is a closed subinterval in $[0,1]$.  A \textit{horizontal subsurface} in $\Sigma \times [0,1]$ is a surface of the form $F \times \{a\}$ where $F \subset \Sigma$ is compact subsurface and $a \in [0,1]$.  We will say that a compact, properly embedded surface $S \subset \Sigma \times [0,1]$ is \textit{flat} if $S$ is the union of a collection of vertical annuli and horizontal subsurfaces such that any two vertical annuli in $S$ are disjoint.  In particular, the vertical annuli have pairwise disjoint boundary loops that coincide with the boundary loops of the horizontal surfaces.

In our case, $N$ will be the complement in $M$ of a regular neighborhood $U$ of the one-sided surface $O$ and a regular neighborhood $H$ of a spine $K$ for the handlebody bounded by $O$. By a \textit{spine}, we mean that the handlebody deformation retracts onto the graph $K$.  We will assume that $S$ is the intersection with $N$ of a closed surface $T$ that is transverse to $O$ and $K$, so that $T \setminus N$ consists of disks in $U$ and annuli and M\"obius bands in $H$. 

Let $\pi$ be the projection map from $N = \Sigma \times [0,1]$ to $[0,1]$.  If $S$ is any properly embedded, piecewise-linear surface in $N$ then (after isotoping $S$ slightly if necessary) the level sets of $f|_S$ will consist of simple closed curves and finitely many graphs in $S$.  We can isotope $S$ so as to make a regular neighborhood of each graph horizontal.  The complement of the horizontal subsurfaces will be foliated by simple closed curves, and thus consist of annuli, which can be isotoped to vertical annuli.  The result of this isotopy is a flat surface, so we have:

\begin{Lem}
\label{makeflatlem}
Every pl surface properly embedded in $\Sigma \times [0,1]$ is isotopic to a flat surface.
\end{Lem}

We will say that two horizontal subsurfaces of $S$ are \textit{adjacent} if there is a vertical annulus with one boundary loop in each of the horizontal subsurfaces.  We will say that a vertical annulus $A$ adjacent to a horizontal subsurface $F$ \textit{goes up} if $A$ is locally above $F$ (with respect to $\pi$).  Otherwise, we will say that $A$ \textit{goes down} from $F$.

Each horizontal subsurface has two potential transverse orientations, which project to vectors in $[0,1]$ pointing opposite ways. A transverse orientation on one horizontal subsurface defines a transverse orientation on each adjacent horizontal subsurface. (If $S$ is non-orientable then this cannot be extended consistently to the whole surface, but is still well defined locally.) If the transverse orientations project to vectors pointing the same way in $[0,1]$ then we will say the adjacent horizontal subsurfaces \textit{face the same way}. Otherwise, we will say the subsurfaces \textit{face opposite ways}.

\begin{Def}
A flat surface $S$ is \textit{tight} if no horizontal annulus has one adjacent vertical annulus going up and the other going down and if $F$ is a horizontal disk or a horizontal annulus with both annuli going the same direction then the closest adjacent subsurface faces the opposite way from $F$.
\end{Def}

If $F$ is a horizontal annulus with one adjacent vertical annulus going up and the other going down then we can shrink $F$ to a loop and isotope the adjacent annuli so that the two vertical annuli form a single vertical annulus.

If $F$ is a horizontal disk or an annulus with both vertical annuli going the same way such that the closest adjacent horizontal subsurface $F'$ faces the same way as $F$, then the projections of $F$ and $F'$ will be disjoint. We can push $F$ into the same level as $F'$, shrinking the annulus (or annuli) between them down to a loop, and pushing any other portions of $S$ out of the way. The resulting surface is still flat, but has one fewer horizontal subsurfaces. Thus we have the following:

\begin{Lem}[Lemma~18 in~\cite{me:upperbnd}]
\label{maketightlem}
Every flat surface $S$ is isotopic to a tight surface $S'$.
\end{Lem}

We will often want to rule out horizontal disks and annuli in tight surfaces.  This is not possible in general, but the following Lemma will allow us to rule them out in many cases. For the proof, see Lemma~19 in~\cite{me:upperbnd}.

\begin{Lem}
\label{tightdisknotesslem}
If $S$ is a tight surface in which some horizontal subsurface is either a disk or an annulus then then there is a horizontal loop that is trivial in $\Sigma$ and essential in $S$.
\end{Lem}

\section{Sweep-outs}
\label{sweepsect}

Note that one boundary component of $N = \Sigma \times [0,1]$ is adjacent to the regular neighborhood $U$ of $O$ and the other component is adjacent to the regular neighborhood $H$ of the spine $K$. 

\begin{Lem}
\label{essdistlem}
Let $T \subset M$ be a surface such that $S = T \cap N$ is an essential flat surface. If $T \cap H$ is a non-empty collection of essential disks and $T \cap U$ is a non-empty collection of essential annuli and M\"obius bands then $d(O) \leq 1-\chi(T)$.
\end{Lem}

\begin{proof}
Paramaterize $N$ so that $\Sigma_0 = \Sigma \times \{0\}$ is adjacent to the regular neghborhood of $O$ and $\Sigma_1 = \Sigma \times \{1\}$ is the other boundary component. By assumption, the collection of loops $S \cap \Sigma_0$ are boundary loops of annuli and M\"obius bands in a regular neighborhood of $O$. Thus the loops of $S \cap \Sigma_0$ are contained in $\mathcal{O}$. Similarly, the loops of $S \cap \Sigma_1$ are contained in $\mathcal{H}$.

Let $t_1 \leq \dots \leq t_k$ be the levels such that each $\Sigma_{t_i}$ contains a horizontal subsurface of $S$. Let $0 = t'_1 \leq t'_2 \leq \dots \leq t'_{k+1}=1$ be the levels between consecutive horizontal subsurfaces. The boundary of each horizontal subsurface $S \cap \Sigma_{t_i}$ contains the loops in both $S \cap \Sigma_{t'_{i-1}}$ and $S \cap \Sigma_{t'_i}$. Thus loops in consecutive horizontal subsurfaces have disjoint projections into $\Sigma$ and if we choose a loop $\ell_i$ in each $S \cap \Sigma_{t'_i}$, we will find a path in the curve complex for $\Sigma$.

Let $j$ be the smallest index such that some horizontal loop in $S \cap \Sigma_{t'_j} \subset T$ bounds a disk $D \subset T$. Such an index exists because $S \cap \Sigma_1$ is non-empty and contains the boundary of each disk in $T \cap H$. If $D$ intersects $O$ then an innermost disk bounded by $D \cap O$ defines a geometric compression of $O$. As noted above, this implies $d(O) = 0$. We will thus assume this is not the case, so $S \cap )$ is empty and $D$ can be pushed entirely into the handlebody bounded by $\Sigma_{t'_j}$. Thus $\partial D \subset S \cap \Sigma_{t'_j}$ defines a vertex in $\mathcal{H}$. 

The path in $\mathcal{C}(\Sigma)$ defined by the portion of $S$ below level $t'_j$ defines a path from $\mathcal{O}$ to $\mathcal{H}$ of length at most $j$. The loops in $S \cap \Sigma_{t'_{j-1}}$ are essential in $T$ since $t'_j$ was chosen to be minimal. Thus the subsurface $S \cap \Sigma \times [0,t'_{j-1}]$ is an essential subsurface of $T$ and has Euler characteristic greater than or equal to $\chi(T)$. Each horizontal subsurface of $S$ has Euler characteristic $-1$ or less (since there are no horizontal disks or annuli and no loops that are trivial in $S$) so $j$ is at most the negative Euler characteristic of $T$. Because we must allow one more edge from a loop in $S \cap \Sigma_{t'_{j-1}}$ to a loop in $S \cap \Sigma_{t'_j}$, we find that $d(O)$ is at most $1-\chi(T)$.
\end{proof}

\begin{proof}[Proof of Theorem~\ref{mainthm1}]
Let $O$ be a one-sided Heegaard surface and $K$ a spine for the complement $M \setminus O$ as above. Let $T \subset M$ be either an incompressible two-sided surface or a one-sided Heegaard surface. In the second case, replace $T$ with a surface results from maximally compressing the original surface. In either case, we can assume that $T$ is geometrically incompressible and transverse to $O$ and $K$.  

Because $T$ is geometrically incompressible, any loop of $O \cap T$ that is trivial in $O$ must be trivial in $T$. The complement $M \setminus O$ is irreducible so we can isotope $T$ so that every loop in $T \cap O$ is essential in $O$. We will further isotope $T$ so that $T \cap (K \cup O)$ is minimal over all such representatives of $T$. Moreover, if $T$ is incompressible, we will choose $T$ to be an incompressible surface in $M$ of the same genus as the original $T$ intersects $(K \cup O)$ minimally over all such incompressible surfaces.

Let $N$ be the complement of a regular neighborhood of $U \cup H$ of $O \cup K$ such that $T$ intersects $H$ in essential disks and intersects $U$ in essential annuli and M\"obius bands. By Lemmas~\ref{makeflatlem} and~\ref{maketightlem}, we can isotope $S = T \cap N$ to a tight flat surface. By Lemma~\ref{tightdisknotesslem}, any horizontal loop in $S$ that is trivial in some $\Sigma \times \{t\}$ must bound a compressing disk $D$ for $S$. Because $T$ is geometrically incompressible (in $M$), such a loop must bound a disk $E \subset T$ such that $D \cap (K \cup O)$ is non-empty. If $T$ is incompressible then replacing the disk $E$ with the compressing disk $D$ creates a new incompressible surface of the same genus. If $T$ is a one-sided Heegaard surface then the complement of $T$ is irreducible so replacing $E$ with $S$ creates a new one-sided surface that is isotopic to $T$ (within $M$). In either case, the new surface intersects $K \cup O$ in strictly fewer components. By assumption, the intersection has already been minimized, so $S$ must be an essential flat surface.

If $T$ is disjoint from $U$ then $T$ is contained in the handlebody $M \setminus U$. However, the only geometrically incompressible surface in a handlebody is a sphere. Similarly, if $T$ is disjoint from $H$ then $T$ is contained in an interval bundle, in which case the only geometrically incompressible surface is the midpoint surface or its double. The double of the midpoint surface is compressible (it bounds a handlebody) and the midpoint surface is $O$. Thus if $T$ is not isotopic to $O$, it must intersect both $\Sigma_0$ and $\Sigma_1$ and the bound follows from Lemma~\ref{essdistlem}. If $T$ is isotopic to $O$ then we may have replaced $T$ at the beginning of the proof with the result of compressing the original one-sided Heegaard surface $T$ some number of times. By Lemma~\ref{mainlem}, this implies that the original $T$ is a stabilization of $O$.
\end{proof}

\section{The complex of surfaces}
\label{surfacecplxsect}

The complex of surfaces $\mS(M)$ was introduced in~\cite{me:calc} as a tool for describing thin position arguments. Two surfaces are \textit{sphere-blind isotopic} if they are related by a combination of isotopy, adding or removing trivial spheres and attaching tubes from sphere components to other components of the surface. The vertices of $\mS(M)$ are sphere-blind isotopy classes of transversely oriented, strongly separating surfaces, with edges connecting each surface to the surfaces that result from compressing it. A collection of $n$ pairwise disjoint compressing disks for a surface define an $n$-dimensional cell whose vertices are the different surfaces that result from compressing along subsets of the collection of disks. See~\cite{me:calc} for a more careful definition.

The proof of Lemma~\ref{mainlem} uses an extension of the complex of surfaces $\mS(M)$, which we will define as follows: Let $\mB \subset \mS(M)$ be the set of vertices such that $v \in \mB$ if and only if the positive component of a surface $S$ representing $v$ is homeomorphic to a twisted interval bundle.  

We will say that $D'$ is a \textit{double disk} for $S$ if $D'$ is a disk made up of a compressing disk $D$ on the negative side of $S$ and an annulus on the positive side.  Such a disk will intersect $S$ in the both the boundary of $D'$ and a second loop in the interior of $D'$.  We will say that $D$ is the compressing disk \textit{associated to} $D'$.  If $S'$ is the result of compressing $S$ along $D$ then $D'$ will be a compressing disk for $S'$.  

The compression from $S$ to $S'$ corresponds to an edge below $v$ pointing towards $v$.  The disk $D'$ is on the positive side of $S'$, so it corresponds to an edge below the vertex $v'$ represented by $S'$ pointing away from $v'$.  Let $v''$ be the lower endpoint of this second edge.  If we add into $\mS$ an edge from $v$ to $v''$ then we can also add a triangle consisting of this new edge and the edges from $v$ to $v'$ and from $v'$ to $v''$  Let $\mS^O = \mS^O(M)$ be the result of adding such an edge and a triangle into $\mS$ for every double disk for a surface representing each vertex in $\mB$.  The new edge will point away from from $v$ and we will call it a \textit{double edge}.  (The original edges will be called \textit{single edges}.) A double edge has both endpoints in the set $\mB$ and represents a (geometric) compression of the one-sided midpoint surface.

Define a \textit{one-sided splitting path} in $\mS^O$ as an oriented path from $v_-$ to any vertex in $\mB$ along single edges.  (In fact, this is a path in $\mS$.)  The local maxima and minima of such a path define a sequence of closed, embedded surfaces $S_0,\dots,S_k \subset M$ that satisfy the following properties:

\begin{enumerate}
    \item The surface $S_0$ is empty.
    \item Between consecutive surfaces $S_i$, $S_{i+1}$ there is a compression body.
    \item The final surface $S_k$ bounds a twisted interval bundle on its positive side.
\end{enumerate}

If there are no internal maxima in the path, i.e.\ the path is monotonic increasing, then there is a single handlebody bounded by the final surface $S_k$.  This handlebody extends to a handlebody whose positive boundary double covers the one-sided midpoint surface defined by the twisted interval bundle so the midpoint surface is a one-sided Heegaard surface for $M$.  Conversely, every one-sided Heegaard surface is represented by a monotonic path from $v_-$ to a vertex in $\mB$.

Lemma~\ref{mainlem} is relatively easy in the case when the compressing disks for $O$ are not nested. The difficulty arises when the compressions are nested, but the proof below uses the complex of surfaces to rearrange a possibly nested sequence of compressions into a disjoint sequence.

\begin{proof}[Proof of Lemma~\ref{mainlem}]
Let $P_0$ be a monotonic, increasing one-sided splitting path for $O$, ending at a vertex $v$. Let $O = O_0, O_1,\dots,O_k = O'$ be a sequence of one-sided surfaces that result from repeatedly compressing $O$. Let $g_1$ be the double edge representing the geometric compressing disk for $O_o = O$ that produces $O_1$. By definition, $g_1$ shares a face with a pair of edges $e_0$, $e'_0$ where $e_0$ is below $v$ and $e'_0$ is below the second endpoint of $e_0$. The edge $e_0$ faces up, so by the barrier axiom for $\mS$, there is a path equivalent to $P_0$ that ends with the edge $e_0$. Replace $P_0$ with this path.
\begin{figure}[htb]
  \begin{center}
  \includegraphics[width=3.5in]{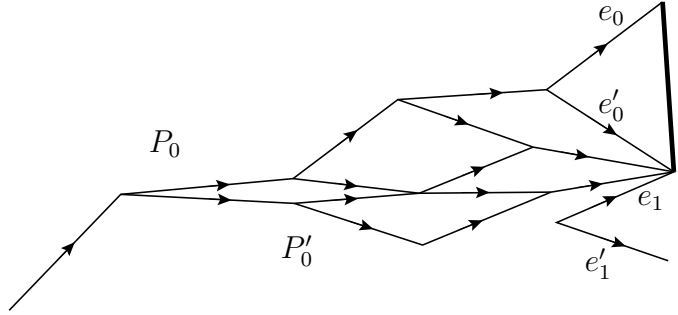}
  \put(-200,60){$P_0$}
  \put(-150,20){$P'_0$}
  \put(-30,110){$e_0$}
  \put(-30,75){$e'_0$}
  \put(-15,40){$e_1$}
  \put(-35,15){$e'_1$}
  \caption{The path $P'_0$ results from replacing the final edge $e_1$ in $P_0$ with $e'_1$, then weakly reducing as much as possible. The thickers line is a double edge.}
  \label{redpathfig}
  \end{center}
\end{figure}

If we replace the edge $e_0$ in $P_0$ with $e'_0$ then we again have an oriented path, this time with a single maximum. Let $P'_0$ be the result of weakly reducing $P_0$ as much as possible, as shown in Figure~\ref{redpathfig}. The path $P'_0$ ends at a vertex representing $O_1$. Let $e_1$, $e'_1$ be the (single) edges defined by the double edge from $O_1$ to $O_2$.  Because the path $P'_0$ is thin and $\mS$ satisfies the Casson-Gordon axiom, the last minimum of $P'_0$ (or the initial vertex if there are no internal minima) is incompressible. Thus the barrier axiom again implies that $P'_0$ is equivalent to a path $P_1$ whose final edge is $e_1$. Replace the edge $e_1$ in $P_1$ with $e'_1$ and let $P'_1$ be the result of weakly reducing this path as much as possible.

Repeat this process for each geometric compression in the sequence from $O_0$ to $O_k$. Define $P_k = P'_{k-1}$. By assumption, $O_k = O'$ is a one-sided Heegaard surface, so the path $P_k$ is a strongly irreducible path for a handlebody. The only strongly irreducible oriented path in a handlebody is a monotonic path, so $P_k$ has no internal maxima.

The sequence of weak reductions that turn the path $P_0$ into $P_k$ define a fan of 2-cells in $\mS(M)$, as in~\cite{me:calc}. Moreover, following~\cite{me:calc}, we can remove diamonds in the fan in which all the edges point up. This implies that in the resulting fan every vertex is the lower endpoint of at least one edge pointing towards that vertex. By following such edges up from the vertex $v_k$ representing $O_k$, we can find a reverse-oriented path from $v_k$ to some vertex $v'$ in the original path $P_0$, as in Figure~\ref{redfinalfig}. Appending the reverse of this path to the reverse of the subpath of $P_0$ from $v'$ to $v_0$ produces a path $P''$ of compression from $v_0$ (the double of $O$) through $v'$ to $v_k$ (the double of $O'$) in which the first half of the compressions are on one side and the second half are on the other side.
\begin{figure}[htb]
  \begin{center}
  \includegraphics[width=4.5in]{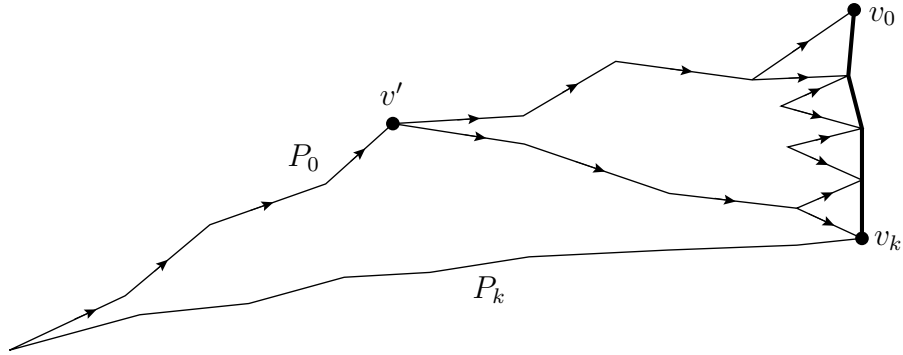}
  \put(-220,70){$P_0$}
  \put(-150,20){$P_k$}
  \put(-185,93){$v'$}
  \put(0,125){$v_0$}
  \put(2,40){$v_k$}
  \caption{The path $P'_0$ results from replacing the final edge $e_1$ in $P_0$ with $e'_1$, then weakly reducing as much as possible. The thickers line is a double edge.}
  \label{redfinalfig}
  \end{center}
\end{figure}

The vertex $v'$ determines a Heegaard surface $T$ for the negative side of $v_k$, which is a handlebody by assumption.  By the classification of Heegaard surfaces in handlebodies (which follows from Waldhausen's Theorem~\cite{wald} and Casson-Gordon's Theorem~\cite{cg}, as in~\cite{st:crossi}), $T$ is a multiple stabilization of a boundary parallel surface, so there is a collection $R$ of reducing spheres, each of which intersects $T$ in a single essential loop cutting off a single unknotted torus.

In the ambient manifold $M$, the surface $T$ bounds a submanifold $M'$ that results from attaching handles to a regular neighborhood of $O_k$. Moreover, we can get from $M'$ to a regular neighborhood of $O$ by drilling out tubes. Let $T'$ be the result of pushing $T$ into $M'$, then attaching the appropriate tubes.  This surface cuts $M'$ into a surface-cross-interval and an interval bundle with one-handles attached. The reducing spheres $R$ intersect $M'$ in essential disks, so by an adaptation of Haken's method similar to that in~\cite{cg:red}, we can isotope each essential disk to intersect the product side of $T'$ in an annulus and the non-product side in a disk. 

Thus we can isotope $O$ to intersect each reducing sphere in an essential loop, and intersect each ball bounded by a reducing sphere in a punctured torus. It follows that $O$ is a multiple-stabilization of $O'$.
\end{proof}

\section{Two-sided Heegaard surfaces}
\label{twosidesect}

Let $O$ be a one-sided Heegaard surface, and $U$ a closed regular neighborhood of $O$ as above. Let $\alpha \subset U$ be an interval in the I-bundle structure on $U$ and let $X \subset U$ be a regular neighborhood of $\alpha$. The complement $H^- = U \setminus X$ is an interval bundle over a surface with boundary and thus a handlebody. By assumption, the complement $M \setminus U$ is a handlebody, which intersects the ball $X$ in two disks. Thus the closure of $(M \setminus U) \cup X$ is a second handlebody $H^+$ with $\partial H^- = \partial H^+$. Their common boundary $\Sigma'$  is a Heegaard surface and as in the introduction, we will say that $\Sigma'$ is \textit{induced by} $O$.

Let $D^+ \subset X$ be a properly embedded, essential disk in $H^+$ dual to $\alpha$, i.e.\ intersecting the arc in a single point. Let $\beta \subset O$ be a properly embedded arc in the surface $O \setminus X$. The union of the intervals in $U$ that intersect $\beta$ is an essential properly embedded disk $D^- \subset H^-$. The intersection $D^- \cap D^+ \subset \Sigma'$ consists of two points, and any essential loop in $O$ disjoint from $\beta \cup X$ will define an essential loop in $\Sigma$ disjoint from $D^-$ and $D^+$.  Thus the distance in the curve complex $\mathcal{C}(\Sigma')$ from $D^-$ to $D^+$ is two.

Depending on the one-sided splitting, the $d(\Sigma')$ may be less than two. However, if $d(\Sigma') = 1$ (i.e. $\Sigma'$ is irreducible but weakly reducible) then there will be a two-sided incompressible surface of genus less than that of $\Sigma'$. If $d(\Sigma') = 0$ then there will be a Heegaard surface of genus strictly less than that of $\Sigma'$. Thus Theorems~\ref{mainthm1} and~\ref{2sidedthm1} imply that for $O$ with sufficiently high distance, the induced Heegaard surface $\Sigma'$ has distance exactly two. (As noted in the introduction, this also follows from Li's results~\cite{li}.)

Let $H$ be a regular neighborhood of the spine of the complement of $O$. As in~\cite{me:upperbnd}, we will let $\mS(M, H)$ be the complex of surfaces relative to $H$: Vertices of $\mS(M, H)$ are transverse isotopy classes of surfaces that intersect $H$ in essential disks. Edges correspond to compressions in the complement of $H$ and bridge compressions that change the intersection of the surface with $H$. Higher dimensional cells are defined by disjoint compressions, as in the complex of surface $\mS(M)$.

A second Heegaard surface $T$ for $M$ defines a path $E$ in $\mS(M, H)$.  Following~\cite{me:calc}, we can weakly reduce the Heegaard splitting, replacing $E$ with a path in which the local minima have index zero and the local maxima have index one. As noted in Section~\ref{flatsect}, every index-zero surface can be isotoped to an essential flat surface, which determines a bound on $d(O)$. Since a local minimum will have genus less than that of $\Sigma$, this implies Theorem~\ref{2sidedthm1} whenever $E$ has an interior local minimum.

If $E$ has only a single maximum then we would like to find an essential flat representative for this maximum. In fact, we need to allow a slightly more general possibility. 

We will say that a subset $S \cap F_t$ of a surface $S \subset N$ has a \textit{flipped square} if it consists of the union of a subsurface $F$ of $S$ and a disk $D$ in $S$ such that $F \cap D$ consists of four points.  Moreover, the intersection of the vertical annuli just above $F_t$ with those just below $F_t$ consists of these same four points as in Figure~\ref{flippedfig}.  In such a subsurface, the normal vector to the disk points in the direction opposite the rest of the subsurface.
\begin{figure}[htb]
  \begin{center}
  \includegraphics[width=3.5in]{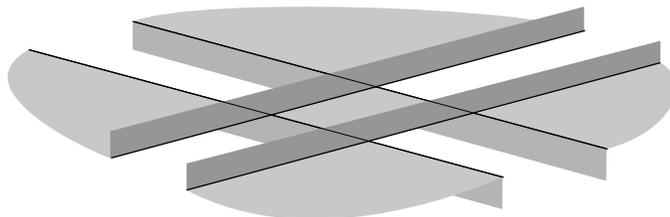}
  \caption{A flipped square is a horizontal disk that intersects the rest of the horizontal subsurface in four points and faces the opposite way.}
  \label{flippedfig}
  \end{center}
\end{figure}

In~\cite{me:upperbnd}, the author showed that every index-one surface can be isotoped to an essential flat surface, possibly with one flipped square in a sweep-out for $M$. In~\cite{me:open}, the author showed that index-two surfaces can also be made essential, with up to two flipped squares, relative to an open book decomposition. Because we have a slightly different setup here than in both those papers, we will prove a new version of these results.

\begin{Lem}
\label{indexesslem}
Every index-one surface in $M$ is isotopic to an essential flat surface, possibly with one flipped square. Ever index-two surface is isotopic to an essential flat surface with up to two flipped squares.
\end{Lem}

The proof uses the following construction:

Let $E_1, E_2 \subset S$ be horizontal subsurfaces of a tight surface $S \subset N$ in consecutive levels of $[0,1]$ such that $E_1$ is below $E_2$ and first assume the subsurfaces face opposite ways.  Let $I \subset [0,1]$ be the interval between the levels containing $E_1$ and $E_2$.  We can think of these as subsurfaces of $F$.  Let $\alpha$ be a properly embedded arc in $E_1$. Because we can project $\alpha$ between pages, we will abuse notation and talk about the intersection of $\alpha$ with subsurfaces in other pages. Assume the vertical annuli that containing the endpoints of $\alpha$ go up from $E_1$ and that $\alpha \cap E_2$ is empty or consists of a regular neighborhood in $\alpha$ of one or both of its endpoints. Then the vertical band $\alpha \times I$ between the levels containing $E_1$ and $E_2$ forms a disk whose boundary consists of an arc in $S$ and a horizontal arc $\alpha'$ which is the closure of $\alpha \setminus E_2$.

The endpoints of $\alpha'$ are contained in one or two vertical annuli or bands above $E_2$.  Assume $I' \subset [0,1]$ is the interval above $I$ such that the level defined by the upper endpoint of $I'$ contains the next highest horizontal subsurfaces $E_3$ of $S$ and that these face the same way as $E_1$. Note that this implies the interior of $\alpha'$ will be disjoint from $E_3$. Then $\alpha' \times I'$ is a disk whose boundary intersects $S$ in two vertical arcs.  The union $D = (\alpha \times I) \cup (\alpha' \times I')$ is a disk whose boundary consists of an arc in $S$ and the horizontal arc parallel to $\alpha'$.  Let $S'$ be the result of isotoping $S$ across the disk $D$ as in Figure~\ref{bandmovefig}, then making this flat surface tight.
\begin{figure}[htb]
  \begin{center}
  \includegraphics[width=5in]{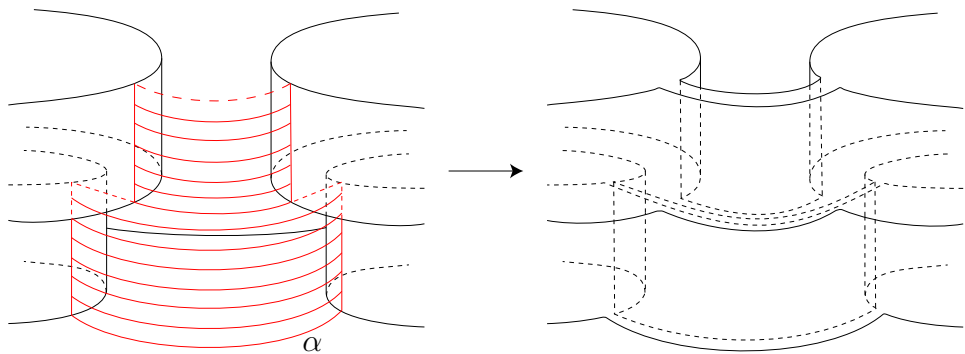}
  \put(-250,0){$\alpha$}
  \caption{The red disk defines a band move from the surface on the left to the surface on the right.}
  \label{bandmovefig}
  \end{center}
\end{figure}

\begin{Def}
We will say that $S'$ is the result of a \textit{band move} on $S$.
\end{Def}

If the horizontal subsurfaces $E_1$, $E_2$, $E_3$ do not face in the directions we have assumed that they do, we can still define a band move, and we refer the reader to~\cite{me:open} for a discussion of these other cases. In our situation, we also allow an additional type of band move that pulls the band across the one-sided surface $O$ and back into $N$.

The proof of Lemma~\ref{indexesslem} relies on the following:

\begin{Lem}
\label{diskisotopylem}
If $D$ is a compressing disk  for a tight surface $S$ then there is a sequence of band moves of $S$ after which we can isotope $D$ to be horizontal.
\end{Lem}

The proof of this follows from an argument almost identical to that in~\cite[Lemma~23]{me:upperbnd}, so we will leave the details to the reader. The proof of Lemma~\ref{indexesslem} combines the proof of Lemmas~15 and~16 in~\cite{me:open}, but is sufficiently different to justify its inclusion here.

\begin{proof}[Proof of Lemma~\ref{indexesslem}]
First consider the index-one case and let $D^-$, $D^+$ be compressing disks for $S$ in different components of the disk complex. By Lemma~\ref{diskisotopylem}, there is a sequence of band moves after which $D^-$ is horizontal in a tight flat surface $S_0$ isotopic to $S$. The disk $D^+$ determines a similar sequence of band moves, defining a sequence $\{S_i\}$ of tight flat surfaces isotopic to $S$, starting with $S_0$. If any one of these surfaces is essential then the proof is complete.

Otherwise, every $S_i$ has a trivial horizontal loop and thus a horizontal compressing disk $D_i$ for $S_i$. If $S_i$ has a horizontal compressing disk disjoint from the band move that produces $S_{i+1}$ then this disk is also horizontal in $S_{i+1}$. Because $D^-$ and $D^+$ are in different components of the disk complex for $S$, there must be some value $i$ such that $D_i$ and $D_{i+1}$ are in different components. Because any two horizontal disks are disjoint, $D_i$ cannot be horizontal in $S_{i+1}$ and $D_{i+1}$ cannot be horizontal in $S_i$. Thus $D_i$ ceases to be horizontal after the band move from $S_i$ to $S_{i+1}$ and $D_{i+1}$ is not horizontal before the move. 

Assume for contradiction the band move pulls the band across the one-sided surface $O$ and let $S'$ be the surface that results from stopping the band move while the band is in $O$. The surface $S'$ intersects $O$ in a pair of pants with adjacent vertical annuli that are locally on opposite sides of $O$. By assumption, these vertical annuli are essential in $\Sigma_t$, so their boundary loops in $O$ are essential in $O$. Thus if we continue the band move, it will not create any trivial horizontal loops. All the loops disjoint from the band move are essential, so this contradicts the assumption that $S_{i+1}$ has a horizontal compression disk.

Thus the band move must be between two horizontal subsurfaces of $S$. Each trivial horizontal loop intersects the band, so if we stop the band move half way through the move, we will find a surface with one flipped square and no trivial horizontal loops, as in Figure~\ref{flippedbandfig}. 
\begin{figure}[htb]
  \begin{center}
  \includegraphics[width=5in]{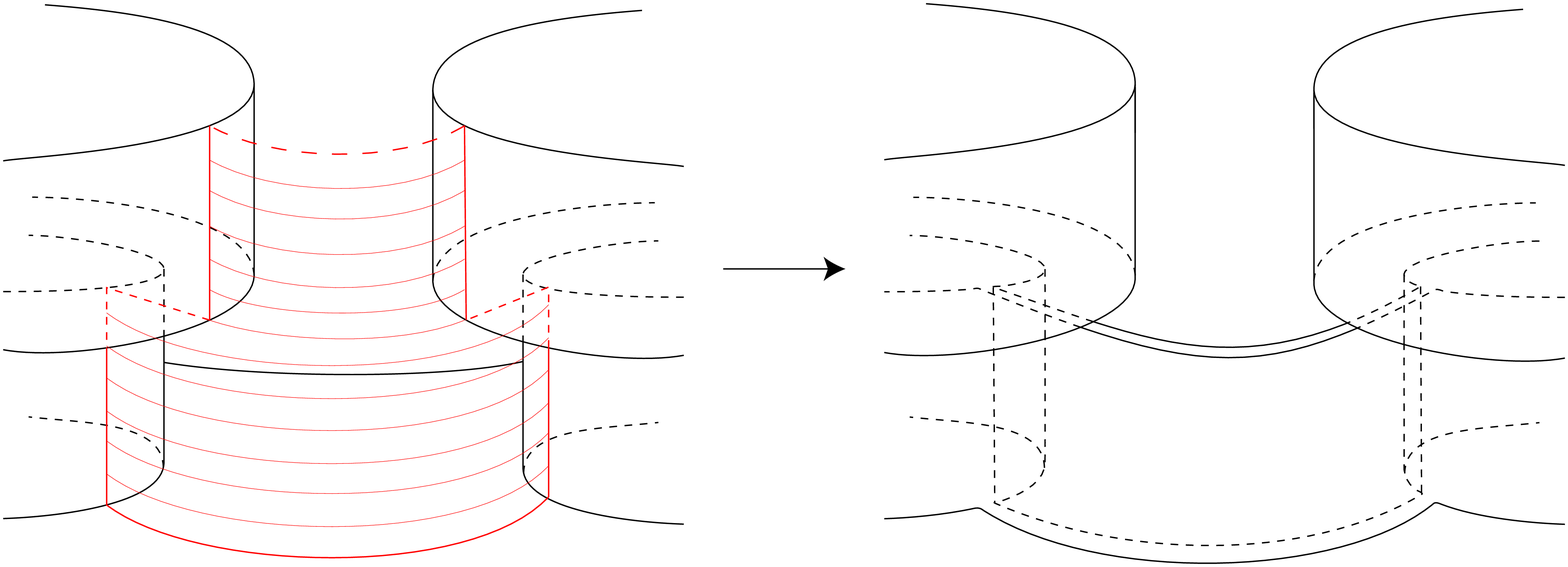}
  \put(-250,0){$\alpha$}
  \caption{Stopping a band move part way through produces a surface with a flipped square.}
  \label{flippedbandfig}
  \end{center}
\end{figure}

Next consider the case when $S$ is an index-two surface. Then the disk complex for $S$ has non-trivial fundamental group, so there is a homotopy non-trivial edge loop consisting of compressing disks $B_0,\dots,B_k$ for $S$. Each disk $B_i$ defines a sequence of band moves between surfaces $S^i_0,\dots,S^i_j$ such that $B_i$ is horizontal in the final $S^i_j$.

Consecutive disks $B_i, B_{i+1}$ are disjoint so the band moves defined by $B_i$, $B_{i+1}$ are disjoint and thus commute with each other.  In other words, performing the first $j$ band moves defined by $B_i$, and the first $k$ band moves defined by $B_{i+1}$ in any order produces the same flat surface, which we will call $S^i_{j,k}$. Under this convention, $S^{i-1}_{0,j} = S^i_j = S^i_{j,0}$. 

If any one of these surfaces $S^i_{j,k}$ is essential, then we have found our essential representative for $S$. Otherwise, we will choose a horizontal disk $D^i_{j,k}$ for each $S^i_{j,k}$. If $j$ and $k$ are the highest indices in the sequences defined by $B_i$, $B_{i+1}$ then we will let $D^i_{j,h} = B^i$ for $h \leq k$ and $D^i_{h,k} = B^{i+1}$ for $h < j$.

By construction, any two surfaces $S^i_{j,k}$, $S^i_{j+1,k}$ or $S^i_{j,k}$, $S^i_{j,k+1}$ are related by a band move. As in the index one case, if there is no horizontal disk common to both surfaces then stopping the band move half way produces an essential surface with one flipped square and we're done. Otherwise, assume that there is a horizontal disk common to every such pair of surfaces, and choose a disk for each pair. To avoid excessive notation, we will not label this disk, but instead say that it is the disk \textit{associated to} the band move.

Every square of surfaces $S^i_{j,k}$, $S^i_{j+1,k}$, $S^i_{j+1,k+1}$, $S^i_{j,k+1}$ thus defines a loop of up to eight disks -- The four disks $D^i_{j,k}$, $D^i_{j+1,k}$, $D^i_{j+1,k+1}$, $D^i_{j,k+1}$ and the four disks associated to the band moves. If there is a horizontal disk $D'^i_{j,k}$ common to all four of the surfaces then $D'^i_{j,k}$ is disjoint from all eight disks. Because the disk complex is flag, the edges from $D'^i_{j,k}$ to these disks defines a collection of triangles forming a disk bounded the loop. If there is such a disk for every square $S^i_{j,k}$, $S^i_{j+1,k}$, $S^i_{j,k+1}$ then the union of all these disks form a single immersed disk whose boundary is the loop $\{B_i\}$, contradicting the assumption that this loop is homotopy non-trivial.

Thus we will assume that there is some square of surfaces $S^i_{j,k}$, $S^i_{j+1,k}$, $S^i_{j,k+1}$, $S^i_{j+1,k+1}$ such that there is no horizontal disk common to all four surfaces. The square is defined by two disjoint band moves between the four surfaces and similarly to the index-one case, stopping both band moves mid-way produces an index-two surface with two flipped squares. Thus $S$ is isotopic to an essential flat surface with up to two flipped squares.
\end{proof}

\section{A two-sided distance bound}
\label{twosideboundsect}

By Lemma~\ref{essdistlem}, if the essential flat surface that results from Lemma~\ref{diskisotopylem} has no flipped squares then we have the desired distance bound. Otherwise, we need a version of this lemma that includes flipped squares.

\begin{Lem}
\label{essdist2lem}
Let $T \subset M$ be a surface such that $S = T \cap N$ is an essential flat surface with zero or more flipped squares. If $T \cap H$ is a non-empty collection of essential disks and $T \cap U$ is a non-empty collection of essential annuli and M\"obius bands then $d(O) \leq 1-\chi(T)$.
\end{Lem}

\begin{proof}
As in the proof of Lemma~\ref{essdistlem}, we will let $t_1,\dots,t_k$ be the levels such that each $\Sigma_{t_i}$ contains a horizontal subsurface of $S$. Let $0 = t'_1,t'_1,\dots,t'_{k+1}=1$ be the levels between consecutive horizontal subsurfaces. We need only check that in the case when a horizontal subsurface $S \cap \Sigma_{t_i}$ contains one or more flipped squares, the distance from $S \cap \Sigma_{t_{i-1}}$ to $S \cap \Sigma_{t_i}$ is at most $-\chi(S \cap \Sigma_{t_i})$.

For each flipped square in $S$, there are two surfaces $S^-$, $S^+$ that come from ``resolving'' the flipped square in $S$, by either completing the band move that created it, or undoing the beginning of the band move. If either of $S^-$, $S^+$ is an essential flat surface then we can apply the argument to this surface, which will have one fewer flipped squares. Thus we can assume that each of $S^-$, $S^+$ contains an inessential vertical annulus.

To find the horizontal loops of $S_-$ or $S_+$, we pinch two parallel sides of the flipped square together, then isotope the resulting loops to remove the two resulting bigons. If both $S_-$ and $S_+$ are inessential then pinching along each pair of parallel edges must create a trivial loop. Pinching a loop to itself along a non-trivial arc cannot produce a trivial loop. If pinching two loops together creates a trivial loop then the two original loops must be parallel. Thus the loops involved in the flipped square consist of two pairs of parallel loops so that each type of loop intersects the flipped square in exactly one arc.

We conclude that every loop in $S \cap \Sigma_{t'_i}$ is disjoint from every loop in $S \cap \Sigma_{t'_{i+1}}$, except for pairs of parallel loops with intersection number at most the number $n$ of flipped squares, making the curves distance at most $1 + \log_2(n) \leq n + 1$. On the other hand, the Euler characteristic of $S \cap \Sigma_{t_i}$ is less than the $-(1 + n)$ so the total distance from $S \cap \Sigma_0$ to $S \cap \Sigma_1$ is at most $1 - \chi(S)$.
\end{proof}

\begin{proof}[Proof of Theorem~\ref{2sidedthm1}]
Let $T'$ be a Heegaard surface for $M$ with genus $g$ such that $2g \leq d(O)$. Let $E$ be the result of weakly reducing (and possibly destabilizing) a path in $\mS(M, H)$ representing $T'$. If $E$ has a local minimum then let $T$ be a surface representing this minimum. In this case, $T$ is incompressible and has genus strictly less than that of $T'$, so the genus bound follows from Theorem~\ref{mainthm1}.

Otherwise, let $T$ be the unique local maximum in the path $E$. Then $T$ is index one and either isotopic to $T'$ or a destabilization of $T'$. By Lemma~\ref{indexesslem}, the intersection $T \cap N$ is isotopic to an essential flat surface, possibly with one flipped square. The surface $T$ has genus at most $g$ so by Lemma~\ref{essdist2lem}, if $T$ intersects both $U$ and $H$ non-trivially then $d(O) \leq 1 - \chi(T) \leq 2g - 1$. Thus $T$ must be disjoint from one of these surfaces.

If $T$ is disjoint from $U$ then $T$ is contained in the handlebody $M \setminus U$.  Because the boundary of a handlebody is compressible, it contains no closed incompressible surfaces (other than trivial spheres) or index-one surfaces, as proved by Casson-Gordon~\cite{cg}. (In fact, it follows from~\cite{bach:index} that no surface in a handlebody has a well-defined index.) Note that the minimal Heegaard surface for a handlebody has disks on only one side, so the disk complex for this surface is contractible.

If $T$ is disjoint from $H$ then it is contained in an interval bundle $N'$ isotopic to $U$. The only closed incompressible surfaces in an interval bundle are the one-sided midpoint surface and a boundary-parallel surface. By Casson and Gordon's Theorem~\cite{cg:red}, if we completely compress a strongly irreducible surface in either direction, the result is either empty or an incompressible surface. Since the only two-sided incompressible surface is boundary parallel, every index-one surface in $N'$ is a Heegaard surface. 

Let $D_1$ be a compressing disk for $S$ on the side of the surface that contains $\partial N'$. As in Lemma~\ref{indexesslem}, $D_1$ defines a sequence of band moves that make the disk horizontal. If the first band move produces an essential surface then we let $S_1$ be this surface and continue performing band moves until we find a surface $S_k$ with a horizontal compressing disk (not necessarily isotopic to $D_1$). Let $S_{k+1}$ be the result of compressing $S_k$ along this horizontal disk. If $S_{k+1}$ is compressible to the side containing $\partial N'$, let $D_2$ be a compressing disk on this side of $S_{k+1}$ and perform the sequence of band moves defined by $D_2$, until we can again compress along a horizontal disk.

If we continue in this way, we will eventually produce an incompressible surface $S_n$. Since the compressions were on the same side as $\partial N'$, the surface $S_n$ must be isotopic to $\partial N'$. We can reconstruct $S$ from $S_n$ by attaching a vertical tube dual to the final compression, then isotoping the surface around, possibly attaching another vertical tube dual to the previous compression and so on. Attaching the first vertical tube produces a Heegaard surface for $N'$. By Proposition~22 in~\cite{me:stabs}, if we attach more tubes to this Heegaard surface, then the resulting Heegaard surface is a stabilization of the original, and thus is not strongly irreducible. Thus we construct $S$ from $S_n$ by attaching a single tube, and $S$ is the standard Heegaard surface for $N'$. This standard Heegaard surface is isotopic to $\Sigma$, so the original Heegaard surface $\Sigma'$ is a either isotopic to $\Sigma$ or a stabilization of $\Sigma$.
\end{proof}

\section{Mapping class groups}
\label{mcgsect}

The Heegaard surface $\Sigma'$ induced by $O$ is defined by an arc $\alpha \subset U$ that intersects $O$ in a single point $p \in O$. Any path from $p$ to itself in $O$ defines an element of $\pi_1(O)$, but it also defines an isotopy of $\alpha$ within $U$ (possibly permuting its endpoints), which in turn defines an isotopy of $\Sigma$. Thus there is a homomorphism from $\pi_1(O)$ into $Isot(M, \Sigma)$. We will show that this homomorphism is the inverse of an isomorphism from $Isot(M, \Sigma)$ onto $\pi_1(O)$.

\begin{proof}[Proof of Theorem~\ref{2sidedthm2}]
As noted in~\cite{me:calc}, every element of the isotopy subgroup $Isot(M, \Sigma')$ of $Mod(M, \Sigma')$ is defined by a path in $\mS(M, H)$ in which the local maxima have index two, the local minima have index zero or one, and all the surfaces are isotopic to $\Sigma'$ in $M$.

If the path has an index-two maximum, let $T$ be an essential flat representative, possibly with one or two flipped squares, of this index-two vertex. By Lemma~\ref{essdist2lem}, this surface must be disjoint from either $U$ or $H$. However, as noted in the proof of Theorem~\ref{2sidedthm1}, $T$ cannot be disjoint from $U$. Moreover, if $T$ is disjoint from $H$ then it is either parallel to $\partial U$ or a standard Heegaard surface for the interval bundle $M \setminus H$. These surfaces have index zero and one, respectively, so $T$ cannot be either of these. Thus there cannot be an index-two maximum in the sequence of paths defining the element of $Isot(M, \Sigma')$.

This implies that the isotopy keeps $\Sigma'$ disjoint from the handlebody $H$ and contained in an interval bundle $U' = M \setminus H$. The compression body component of $U' \setminus \Sigma'$ intersects the one-sided surface $O$ in a single disk $E$ and any (ambient) isotopy of $\Sigma'$ induces an isotopy of $O$ that moves $E$ around in $O$ then returns $E$ to itself. The path of a base point in $E$ defines a unique element of $\pi_1(O)$, determining a homomorphism from $Isot(M, \Sigma')$ into $\pi_1(O)$. Every element of $\pi_1(O)$ extends to an isotopy of $E$ in $O$, which can be further extended to an element $Isot(M, \Sigma')$, so this homomorphism is onto. 

To check that the homomorphism is one-to-one, note that any isotopy that fixes the base point in $E$ can be conjugated to fix $E$ setwise. Any isotopy of $\Sigma'$ that fixes the compressing disk $E$ induces an isotopy of the surface $\Sigma$ that results from compressing along $E$. However, $\Sigma$ is boundary parallel in $U'$ so there are no non-trivial isotopies of $\Sigma$. The isotopy must be trivial on $\Sigma'$, so the homomorphism is one-to-one and $Isot(M, \Sigma)$ is isomorphic to $\pi_1(O)$ by the expected map.
\end{proof}

\bibliographystyle{amsplain}
\bibliography{oneside}

\end{document}